\documentclass[12pt,a4paper]{amsart}
\usepackage{mathrsfs}
\usepackage{amsfonts}
\usepackage{txfonts}

\usepackage{hyperref}
\usepackage{latexsym}
\usepackage{amssymb}
\allowdisplaybreaks

\newtheorem{theorem}{Theorem}[section]
\newtheorem{lemma}[theorem]{Lemma}

\theoremstyle{definition}

\newtheorem{remark}{Remark}[section]

\numberwithin{equation}{section}

\begin{document}

\title[H\"{o}lder property of the resolvent of a monotone operator in Banach spaces ]
{{\bf H\"{o}lder property of the resolvent of a monotone operator in Banach spaces}}

\author{ Changchi Huang, Jigen Peng, Yuchao Tang$^*$ }

\address{ Changchi Huang:  School of Mathematics And Information Science, Guangzhou University,
Guangzhou, 510006, China}
\email{944629901@qq.com\;\;(C. Huang)}
\address{ Jigen Peng:   School of Mathematics And Information Science, Guangzhou University,
Guangzhou, 510006, China}
\email{jgpeng@gzhu.edu.cn\;\;(J.P. Peng)}
\address{ Yuchao Tang$^*$:   School of Mathematics And Information Science, Guangzhou University,
Guangzhou, 510006, China}
\email{yctang@gzhu.edu.cn\;\;(Y.C. Tang)}

\thanks{$^*$ The corresponding author. This work was
supported by the National Natural Science Foundations of China (12031003, 12571558, 12571491), the Guangzhou Education Scientific Research Project 2024 (202315829) and the Jiangxi Provincial Natural Science Foundation (20224ACB211004), and the Postdoctoral Startup Foundation of Guangdong Province (62402153). }


\begin{abstract}
  Let $E$ be a Banach space, and let \(J: E \to E^{*}\) denote the normalized duality mapping. In this paper, we establish an upper bound for \( \|Jx - Jy\| \) in \( q \)-uniformly smooth Banach spaces, where the bound is expressed in terms of a relatively simple function of \(\|x - y\|\).  Subsequently, we derive the H\"{o}lder property of mappings of firmly nonexpansive type in 2-uniformly convex and $q$-uniformly smooth Banach spaces ($1<q\leq 2$). As an application,
  we apply this result to the resolvent of a monotone operator in Banach spaces.
\end{abstract}

\keywords{ H\"{o}lder property; firmly nonexpansive mapping; resolvent; monotone operator; Banach space}

\subjclass[2010]{Primary 47H10; Secondary 47H05; 46B20}

\maketitle

\section{Introduction}

Let \( E \) be a smooth Banach space and \(C\) be a nonempty subset of $E$. Kohsaka and Takahashi \cite{KT} introduced the definition of a firmly nonexpansive type mapping \( T: C \to E \), which is defined by
\begin{align}\label{fnt}
  \langle Tx - Ty, JTx - JTy \rangle \leq \langle Tx - Ty, Jx - Jy \rangle,~~\quad \forall x,y\in C,
\end{align}
where \(J: E \to E^{*}\) is the normalized duality mapping. They \cite{KT} proved that if \( E \) is a smooth, strictly convex, and reflexive Banach space, \( A \subset E \times E^* \) be a monotone operator satisfying $D(A) \subset C \subset J^{-1}R(J + rA)$ for all \( r > 0 \), and \( r > 0 \), then the resolvent \( J_{r} = (J + rA)^{-1}J \) of \( A \) is of firmly nonexpansive type (see Lemma \ref{risfnt}). In fact, mappings of firmly nonexpansive type reduce to firmly nonexpansive mappings in Hilbert spaces, since the duality mapping $J$ in a Hilbert space coincides with the identity operator $I$ (see \cite{BC}).

From Koji and Kohsaka \cite{AK}, we can get that if \( E \) be a real smooth and 2-uniformly convex Banach space, then there exists \( \mu \geq 2 \) such that
\begin{align}\label{J inequa1}
  \frac{1}{\mu}\| x - y\| ^{2} \leq \langle x - y, Jx - Jy \rangle
\end{align}
for all \( x, y \in E \) (see Lemma \ref{phiinequ1} and Lemma \ref{main inequa1}). Substituting \(x\) and \(y\) with \(Tx\) and \(Ty\) respectively in formula (\ref{J inequa1}), we obtain that there exists a constant \( \mu \geq 2 \) such that
\begin{align}\label{J inequa2}
  \langle Tx - Ty, JTx - JTy \rangle \geq \frac{1}{\mu} \|Tx - Ty\|^2.
\end{align}
Moreover, since \( \langle Tx - Ty, Jx - Jy \rangle \leq \|Tx - Ty\| \|Jx - Jy\| \),  and together with formula (\ref{J inequa2}), we can derive from inequality (\ref{fnt}) that:

\begin{align}\label{JT inequ}
  \frac{1}{\mu}\|Tx - Ty\| \leq \|Jx - Jy\|.
\end{align}

As far as we know, the inequality of (\ref{JT inequ}) for mappings of firmly nonexpansive type provides only a coarse estimate. If we can establish an upper bound for \(\|Jx - Jy\|\) in terms of a relatively simple function of $\|x-y\|$, we can obtain more precise information about the mapping \(T\). In this paper, we derive such an estimate for \( \|Jx - Jy\| \) in \( q \)-uniformly smooth and 2-uniformly convex Banach spaces (It is known that all the Lebesgue spaces $L_q$ are $q$-uniformly smooth and $2$-uniformly convex whenever $1<q\leq 2.$). The purpose of this result is to establish the H\"{o}lder property of mappings of firmly nonexpansive type. Specifically, we show that for every $R>0$ and for all$x,y \in E$ with $\|x\| \leq R, \|y\|\leq R$, there exist a constant \( L \) and \( 0 < \alpha \leq 1 \) such that \( \|Tx - Ty\| \leq L\|x - y\|^{\alpha} \). To the best of our knowledge, this result is new. Finally, we apply this result to the resolvent of a monotone operator in Banach spaces.

The remainder of the paper is organized as follows. Section 2 reviews preliminary concepts in Banach space geometry, such as the normalized duality mappings and $2$-uniformly smooth Banach spaces. Section 3 presents the main results of the paper. 

\section{Preliminaries}
Let $E$ be a Banach space, and let \(J: E \to E^{*}\) be the normalized duality mapping, which is defined by
\begin{align}\label{dual map}
  J x=\left\{x^{*} \in E^{*}:\left< x, x^{*}\right>=\| x\| ^{2}=\left\| x^{*}\right\| ^{2}\right\}
\end{align}
for all \(x \in E\). A Banach space $E$ is said to be smooth if for every $x\neq 0$ there is a unique supporting functional at $x$, i.e., a unique functional $x^* \in E^*$ such that $\|x^*\| = 1$ and $ \langle x^*, x\rangle = \|x\|$. This supporting functional is denoted by $j(x)$. If $E$ is smooth, then $J(x) = \|x\| j(x)$.

The modulus of convexity of \(E\), is
\[\delta_{E}(\varepsilon)=\inf \left\{1-\frac{\| x+y\| }{2}:\| x\| =\| y\| =1,\| x-y\| =\varepsilon\right\},\]
where the infimum can also be taken over all $\|x\|\leq 1, \|y\|\leq 1$. The
modulus of smoothness is
\[\rho_{E}(\tau)=sup \left\{\frac{\| x+\tau y\| +\| x-\tau y\| }{2}-1:\| x\| =\| y\| =1\right\},\]
where the supremum can also be taken over all $\|x\|\leq 1, \|y\|\leq 1$.

A Banach space $E$ is called uniformly convex if \(\delta_{E}(\varepsilon)>0\) for all \(\varepsilon>0\), and is called $q$-uniformly convex \((2 \leq q<\infty)\) ) if there exists a constant \(C>0\) such that \(\delta_{E}(\varepsilon) \geq C \varepsilon^{q}\) for all \(\varepsilon>0\). $E$ is called uniformly smooth if \(\rho_{E}(\tau) / \tau \to 0\) as \(\tau \to 0\) , and is called p-uniformly smooth \((1<p \leq 2)\) if there exists a constant \(K>0\) such that \(\rho_{E}(\tau) \leq K \tau^{p}\) for all \(\tau>0\).

All symbols used above are standard and can be found in Megginson's book \cite{MRE} , a classic monograph in Banach space theory.


Let $E$ be a smooth Banach space. Following Alber \cite{AY} and Kamimura and Takahashi \cite{KT2}, let \(\phi: E \times E \to \mathbb{R}\) be the mapping defined by
\begin{align}\label{phi}
  \phi(x, y)=\| x\|^{2} - 2\langle x, Jy\rangle + \| y\|^{2}
\end{align}
for all \(x, y \in X\). Note that $\phi$ is the Bregman distance corresponding to \(\|\cdot\|^{2}\).

\begin{lemma}\label{phiequ}\cite{AY,AK}
  Let $E$ be a real smooth Banach space. Then the following identities hold:
   \begin{align}
     \phi(x,y) + \phi(y,x) = 2\langle x-y, Jx-Jy\rangle, \forall x,y\in E.
   \end{align}
\end{lemma}

\begin{lemma}\label{phiinequ1}\cite{AK}
  Suppose that Banach space $E$ is 2-uniformly convex and smooth. Then there exists $\mu \geq 1$ such that
\[\frac{1}{\mu}\| x-y\|^{2} \leq \phi(x, y) \]
for all \(x, y \in E\).
\end{lemma}

\begin{lemma}\label{main inequa1}
  Let $E$ be a real smooth and $2$-uniformly convex Banach space. Then, there exists \(\mu \geq 1\) such that
\[\frac{1}{2\mu}\| x-y\| ^{2} \leq \langle x-y, Jx-Jy \rangle,~~ \forall x, y \in E .\]
\end{lemma}
\begin{proof}
By Lemma \ref{phiequ}, we have
\[
\langle x - y, Jx - Jy \rangle = \frac{1}{2}\phi(x,y) + \frac{1}{2}\phi(y,x) \geq \frac{1}{2}\phi(x,y),
\]
since \( \phi \geq 0 \).

By Lemma \ref{phiinequ1}, there exists \( \mu \geq 1 \) such that \( \phi(x,y) \geq \frac{1}{\mu}\|x - y\|^2 \). Combining these gives
\[
\langle x - y, Jx - Jy \rangle \geq \frac{1}{2\mu}\|x - y\|^2.
\]
\end{proof}

Kohsaka and Takahashi \cite{KT} proved the following result, which shows that the resolvent \( J_{r} = (J + rA)^{-1}J \) of \( A \) is of firmly nonexpansive type.
\begin{lemma}\cite{KT}\label{risfnt}
Let \( E \) be a smooth, strictly convex and reflexive Banach space, let \( C \) be a nonempty closed convex subset of \( E \), and let \( A \subset E \times E^* \) be a monotone operator satisfying $D(A) \subset C \subset J^{-1}R(J + rA)$ for all \( r > 0 \). Let \( r \) be a positive real number and let \( J_r x = (J + rA)^{-1} Jx \) for all \( x \in C \). Then \( J_r : C \to D(A) \) is of firmly nonexpansive type.
\begin{align}
  \langle J_r x - J_r y, J J_r x - J J_r y \rangle \leq \langle J_r x - J_r y, Jx - Jy \rangle.
\end{align}
\end{lemma}
\section{Main results}

In this section, we establish the H\"{o}lder property of mappings of firmly nonexpansive type and then apply it to the resolvent of maximal monotone operators in Banach spaces. First, we derive an upper bound for \( \|Jx - Jy\| \) by using the modulus of smoothness.
\begin{lemma}\label{keylem 1}
  Let $E$ be a uniformly smooth Banach space. Then
  \begin{align}
    \| J(x) - J(y) \| \leq 2\max\{ \| x \|, \| y \| \}  \frac{\rho_E \left( 2 \left\| \frac{x}{\| x \|} - \frac{y}{\| y \|} \right\| \right)}{\left\| \frac{x}{\| x \|} - \frac{y}{\| y \|} \right\|}
  \end{align}
for every $x,y\in E$, $x \neq y$, $x \neq 0$, and $y \neq 0$.
\end{lemma}
\begin{proof}
The norm is convex function, and its derivative at \( x \) is \( j(x) \), hence
\begin{align}\label{j inequal}
\langle j(u), v \rangle \leq \| u + v \| - \| u \|
\end{align}
for every \( u \neq 0 \) and for every \( v \).

Fix $x,y \in E$, and $x \neq 0, y \neq 0$. Let $z\in E$ be any vector with $\|z\| = \big\| \frac{x}{\|x\|} - \frac{y}{\|y\|}\big\|$. Then

\[
\begin{aligned}
\allowdisplaybreaks
&\langle J(y), z \rangle - \langle J(x), z \rangle \\
&= \| y \| \left\langle J\left( \frac{y}{\| y \|} \right), z \right\rangle - \| x \| \left\langle J\left( \frac{x}{\| x \|} \right), z \right\rangle \\
&= \| y \| \left\langle j\left( \frac{y}{\| y \|} \right), z \right\rangle - \| x \| \left\langle j\left( \frac{x}{\| x \|} \right), z \right\rangle \\
&\leq \| y \| \left( \left\| \frac{y}{\| y \|} + z \right\| - 1 \right) + \| x \| \left\langle j\left( \frac{x}{\| x \|} \right), \frac{x}{\| x \|} - \frac{y}{\| y \|} - z \right\rangle (\mathrm{by}\, (\ref{j inequal}))\\
&\leq \| y \| \left( \left\| \frac{y}{\| y \|} + z \right\| - 1 \right) + \| x \| \left( \left\| \frac{2x}{\| x \|} - \frac{y}{\| y \|} - z \right\| - 1 \right) (\mathrm{by}\, (\ref{j inequal}))\\
&= \| y \| \left( \left\| \frac{x}{\| x \|} + \left( \frac{y}{\| y \|} - \frac{x}{\| x \|} + z \right) \right\| - 1 \right) \\
&\quad + \| x \| \left( \left\| \frac{x}{\| x \|} - \left( \frac{y}{\| y \|} - \frac{x}{\| x \|} + z \right) \right\| - 1 \right) \\
&\leq \max\{ \| x \|, \| y \| \} 2\rho_E \left( \left\| \frac{y}{\| y \|} - \frac{x}{\| x \|} + z \right\| \right) (\mathrm{by\, modulus\, of\, smoothness} )\\
&\leq 2\max\{ \| x \|, \| y \| \} \rho_E \left( 2 \left\| \frac{x}{\| x \|} - \frac{y}{\| y \|} \right\| \right)
\end{aligned}
\]
because \( \| z \| = \left\| \frac{x}{\| x \|} - \frac{y}{\| y \|} \right\| \).

Since $z\in E$ is arbitrary,
\[
\| J(x) - J(y) \| \leq 2\max\{ \| x \|, \| y \| \}  \frac{\rho_E \left( 2 \left\| \frac{x}{\| x \|} - \frac{y}{\| y \|} \right\| \right)}{\left\| \frac{x}{\| x \|} - \frac{y}{\| y \|} \right\|}.
\]
\end{proof}

We now use Lemma \ref{keylem 1} to establish an upper bound for \(\|Jx - Jy\|\) in $q$-uniformly smooth Banach spaces, where the bound is given in terms of a relatively simple function of \(\|x - y\|\).
\begin{theorem}\label{main 1}
  Let $E$ be a uniformly smooth Banach space whose modulus of smoothness satisfies $\rho_{E}(\tau) \leq K\tau^q$ for some $1< q \leq 2$. Then, there is a constant $C$ such that
  \begin{align}\label{J inequa3}
    \| J(x) - J(y) \| \leq M  \| x - y \|^{q - 1} \max\{ \| x \|, \| y \| \}^{2 - q}.
  \end{align}
In particular, if \( E \) is a 2-uniformly smooth Banach space, then its duality mapping \( J \) is Lipschitz continuous.
\end{theorem}

\begin{proof}
For the case of  $x= y$ or $x= 0$ or $y= 0$, it is easy to see that the inequality (\ref{J inequa3}) holds. Next, assume that $x \neq y$, $x \neq 0$, and $y \neq 0$, we have
\begin{align*}
\left\| \frac{x}{\| x \|} - \frac{y}{\| y \|} \right\|
&\leq \left\| \frac{x}{\| x \|} - \frac{y}{\| x \|} \right\| + \left\| \frac{y}{\| x \|} - \frac{y}{\| y \|} \right\| \\
&\leq \frac{\| x - y \|}{\| x \|} + \frac{\big| \| x \| - \| y \| \big |}{\| x \|}\\
&\leq \frac{2 \| x - y \|}{\| x \|}.
\end{align*}

Similarly, \( \left\| \frac{x}{\| x \|} - \frac{y}{\| y \|} \right\| \leq \frac{2 \| x - y \|}{\| y \|} \), hence
\begin{align}\label{keyinequ 2}
   \left\| \frac{x}{\| x \|} - \frac{y}{\| y \|} \right\| \leq \frac{2 \| x - y \|}{\max\{ \| x \|, \| y \| \}}.
\end{align}

From Lemma \ref{keylem 1}, and the inequality (\ref{keyinequ 2}), we get
\begin{align*}
\| J(x) - J(y) \| &\leq 2\max\{ \| x \|, \| y \| \}  \frac{\rho_E \left( 2 \left\| \frac{x}{\| x \|} - \frac{y}{\| y \|} \right\| \right)}{\left\| \frac{x}{\| x \|} - \frac{y}{\| y \|} \right\|}\\
&\leq 2\max\{ \| x \|, \| y \| \}  2^q K  \left\| \frac{x}{\| x \|} - \frac{y}{\| y \|} \right\|^{q - 1}\\
&\leq 2^{q+1} K \max\{ \| x \|, \| y \| \}  \left( \frac{2 \| x - y \|}{\max\{ \| x \|, \| y \| \}} \right)^{q - 1}\\
&\leq M \| x - y \|^{q - 1} \max\{ \| x \|, \| y \| \}^{2 - q},
\end{align*}
where $M = 2^{2q}K$.
\end{proof}
\begin{remark}
 Kim and Xu \cite{KX} proved that a Banach space \( E \) is a Hilbert space if and only if its duality mapping \( J \) is Lipschitz continuous with Lipschitz constant 1. By Theorem \ref{main 1}, if a Banach space \( E \) is 2-uniformly smooth, then its duality mapping \( J \) is Lipschitz continuous.
\end{remark}



We now present a theorem concerning H\"{o}lder property of mappings of firmly nonexpansive type.
\begin{theorem}\label{main thm}
  Let $E$ be a $q$-uniformly smooth and $2$-uniformly convex Banach space, $1< q \leq 2$. Let \( C \) be a nonempty closed convex subset of \( E \), \( T \) is of firmly nonexpansive type on $C$. Then for any \( R > 0 \) and any \( x, y \in E \) such that \( \|x\| \leq R\), \( \|y\| \leq R\), then there is a constant $L$ such that the following inequality holds:
  \begin{align*}
    \|Tx - Ty\| \leq L \|x-y\|^{q-1}.
  \end{align*}
\end{theorem}
\begin{proof}
Recall that a mapping \( T: C \to E \) is said to be firmly nonexpansive type if for all \( x, y \in C \),
\begin{align}\label{fnt2}
\langle Tx - Ty, JTx - JTy \rangle \leq \langle Tx - Ty, Jx - Jy \rangle.
\end{align}

Since \( E \) is a \( q \)-uniformly smooth and 2-uniformly convex Banach space, Lemma \ref{main inequa1} guarantees the existence of a constant \( \mu \geq 2 \) such that
\begin{align}\label{JT2}
\langle Tx - Ty, JTx - JTy \rangle \geq \frac{1}{\mu} \|Tx - Ty\|^2.
\end{align}

Moreover, by the Cauchy-Schwarz inequality, we have
\[
\langle Tx - Ty, Jx - Jy \rangle \leq \|Tx - Ty\| \|Jx - Jy\|.
\]
Combining this with inequalities (\ref{fnt2}) and (\ref{JT2}), we substitute the inner product bound into (\ref{fnt2}) and divide both sides by \( \|Tx - Ty\| \) (for \( Tx \neq Ty \); the case \( Tx = Ty \) is trivial) to derive:
\[
\frac{1}{\mu} \|Tx - Ty\| \leq \|Jx - Jy\|.
\]

Finally, for all \( x, y \in C \) with \( \|x\| \leq R \) and \( \|y\| \leq R \), Theorem \ref{main 1} implies
\[
\|Jx - Jy\| \leq MR^{2-q} \|x - y\|^{q-1}.
\]
Substituting this into the inequality above yields the desired result, where $L=\mu MR^{2-q}$. This completes the proof.
\end{proof}

By Lemma \ref{risfnt}, Theorem \ref{main thm} holds for the resolvent \( J_{r}\).

\begin{theorem}\label{main thm2}
Let \( E \) be a $q$-uniformly smooth and $2$-uniformly convex Banach space, $1< q \leq 2$. Let \( C \) be a nonempty closed convex subset of \( E \), and let \( A \subset E \times E^* \) be a monotone operator satisfying $D(A) \subset C \subset J^{-1}R(J + rA)$ for all \( r > 0 \). Let \( r \) be a positive real number and let \( J_r x = (J + rA)^{-1} Jx \) for all \( x \in C \). Then for any \( R > 0 \) and any \( x, y \in E \) such that \( \|x\| \leq R\), \( \|y\| \leq R\), then there is a constant $L$ such that the following inequality holds:
  \begin{align*}
    \|J_{r}x - J_{r}y\| \leq L \|x-y\|^{q-1}.
  \end{align*}
\end{theorem}

\begin{remark}
  If $q =2$, then $J$ is Lipschitz continuous. Therefore, the assumptions  $\|x\| \leq R$ and $\|y\| \leq R$ in Theorem \ref{main thm} and \ref{main thm2} can be omitted. It should be pointed out that $2$-uniformly smooth and $2$-uniformly convex Banach spaces are necessarily isomorphic to Hilbert spaces \cite{BB}.
\end{remark}


\bibliographystyle{amsalpha}

\begin{thebibliography}{AcBeRu}

\bibitem{AY} Y. Alber, Metric and generalized projection operators in Banach spaces: properties and applications. In: Theory and applications of nonlinear operators of accretive and monotone type. New York: Dekker; 1996. 15-50. (Lecture notes in pure and applied mathematics; vol. 178).

\bibitem{AK} K. Aoyama, F. Kohsaka, Strongly relatively nonexpansive sequences generated by firmly nonexpansive-like mappings. Fixed Point Theory Appl. 2014, 2014:95.

\bibitem{BB} B. Beauzamy, Introduction to Banach Spaces and Their Geometry. North-Holland, Amsterdam (1985).
    
\bibitem{BC} H.H. Bauschke, P.L. Combettes, Convex analysis and monotone operator theory in Hilbert spaces. Second edition. Springer, Cham, 2017. xix+619 pp.

\bibitem{KX} T.H. Kim, H.K. Xu, Some Hilbert space characterizations and Banach space inequalities. Math. Inequal. Appl. 1 (1998), 113-121.

\bibitem{KT2} S. Kamimura, W. Takahashi, Strong convergence of a proximal-type algorithm in a Banach space, SIAM J. Optim., 13 (2002), 938-945.

\bibitem{KT} F. Kohsaka, W. Takahashi, Existence and approximation of fixed points of firmly nonexpansive-type mappings in Banach spaces. SIAM J. Optim. 19 (2008), 824-835.

\bibitem{MRE} R.E. Megginson, An introduction to Banach space theory. Graduate Texts in Mathematics, 183. Springer-Verlag, New York, 1998. xx+596 pp.
\end{thebibliography}

\end{document}